\documentclass[]{article}
\usepackage{indentfirst}
\usepackage{amsmath}
\usepackage{authblk}
\usepackage{geometry}
\usepackage{accents}
\usepackage{statex}
\usepackage[english]{babel}
\usepackage{shuffle}

\newtheorem{theorem}{Theorem}
\newtheorem{proposition}{Proposition}
\newtheorem{lemma}{Lemma}
\newtheorem{corollary}{Corollary}

\newenvironment{proof}[1][Proof]{\begin{trivlist}
		\item[\hskip \labelsep {\bfseries #1}]}{\end{trivlist}}
\newenvironment{definition}[1][Definition]{\begin{trivlist}
		\item[\hskip \labelsep {\bfseries #1}]}{\end{trivlist}}
\newenvironment{example}[1][Example]{\begin{trivlist}
		\item[\hskip \labelsep {\bfseries #1}]}{\end{trivlist}}
\newenvironment{remark}[1][Remark]{\begin{trivlist}
		\item[\hskip \labelsep {\bfseries #1}]}{\end{trivlist}}

\title{$\phi$-Thue-Morse sequences and infinite products}
\author{Shuo LI\footnote{shuo.li@imj-prg.fr}}
\date {08/06/2020}

\begin{document}
	
\maketitle

{\footnotesize
\begin{quotation}
{\noindent \bf Abstract.} In this article we introduce a new approach to compute infinite products defined by automatic sequences involving the Thue-Morse sequence.  As examples, for any positive integers $q$ and $r$ such that $0 \leq r \leq q-1$, we find infinitely many couples of rational functions $R(x)$ and $S(n)$ such that
$$\prod_{n=0}^{\infty}R(n)^{\frac{1+a_n}{2}}S(n)^{\frac{1-a_n}{2}}=2cos(\frac{2r+1}{2q}\pi),$$
where $(a_n)_{n \in \mathbf{N}}$ is the Thue-Morse sequence beginning with $a_0=1,a_1= -1$. 
\end{quotation}}

\section{Introduction}

Given an automatic sequence $(s_n)_{n \in \mathbf{N}}$ involving the sum of binary digits of the integers, it is interesting to find classes of rational functions $R$ such that the infinite product $\prod_{n\geq 0} R(n)^{s_n}$ has an expression in terms of known constants. To do so, there are several known approaches. In \cite{Allouche1985}\cite{ALLOUCHE1990}\cite{ALLOUCHE2000}, rational functions $R$ are obtained by computing special values on some particular functions. In\cite{Allouche1989}\cite{ALLOUCHE201595}\cite{HU2016589}\cite{Allouche2019}, authors use combinatorial methods inspired by \cite{Woods}\cite{Robbins}. In this article, we consider the reversal problem: for a given real number $a$, we try to find sequences of rational functions involving the Thue-Morse sequence $(R(n))_{n \in \mathbf{N}}$, such that $\prod_{n \geq 0}R(n)=a$. 

The motivation of this article is to detect the arithmatical nature of some well-known numbers involving the Thue-Morse sequence, for example, the Flajolet-Martin numbers, which is still an open question (see \cite{Allouche2019}). One of the Flajolet-Martin numbers is defined by
$$\prod_{n \geq 1}(\frac{2n}{2n+1})^{a_n},$$
where $(a_n)_{n \in \mathbf{N}}$ is the Thue-Morse sequence beginning with $a_0=1,a_1= -1$. The author states that, rather than calculate directly $\prod_{n \geq 1}(\frac{2n}{2n+1})^{a_n}$, one may work on the infinite product $\prod_{n \geq 1}(1+(\frac{a_n}{2n+1}))$.The last infinite product, inspired by \cite{DILCHER201843}, can be calculated as the limit of a sequence $\prod_{n\geq 1}(1+ (\frac{a^{(i)}_n}{2n+1}))$ such that $(a^{(i)}_n)_{n \in \mathbf{N}}$ is a sequence of periodic sequences converging to the Thue-Morse sequence. Furthermore, there is a natural relation beween these two products:

$$\prod_{n \geq 1}\left(\frac{2n}{2n+1}\right)^{a_n}\prod_{n\geq 1}\left(1+ (\frac{a_n}{2n+1})\right)=\prod_{a_n=1}\left(\frac{2n}{2n+1}\frac{2n+2}{2n+1}\right)\prod_{a_n=-1}\left(\frac{2n+1}{2n}\frac{2n}{2n+1}\right)$$
As a result, 
$$
  \begin{aligned}
        (\prod_{n \geq 1}\left(\frac{2n}{2n+1}\right)^{a_n}\prod_{n\geq 1}\left(1+ (\frac{a_n}{2n+1})\right))^2&=\prod_{a_n=1}\left(\frac{2n}{2n+1}\frac{2n+2}{2n+1}\right)^2\\
        &=\prod_{n=1}^{\infty}\left(\frac{2n}{2n+1}\frac{2n+2}{2n+1}\right)^{a_n}\prod_{n=1}^{\infty}\left(\frac{2n}{2n+1}\frac{2n+2}{2n+1}\right)\\
        &=\prod_{n \geq 1}\left(\frac{2n}{2n+1}\right)^{a_n}\frac{\sqrt{2}}{2}\prod_{n=1}^{\infty}\left(\frac{2n}{2n+1}\frac{2n+2}{2n+1}\right)\\
        &=\prod_{n \geq 1}\left(\frac{2n}{2n+1}\right)^{a_n}\frac{\Gamma(1)}{\Gamma(\frac{1}{2})^2}\frac{\sqrt{2}}{2}\\
        &=\prod_{n \geq 1}\left(\frac{2n}{2n+1}\right)^{a_n}\frac{\pi}{4}\frac{\sqrt{2}}{2}
    \end{aligned}
$$
The third equality is from the famous Wood-Robbins equality \cite{Woods}\cite{Robbins}.

Despite that this approach can still not touch the kernel of the problem, we find closed forms for other infinite products. As results, we prove that for given integers $q$ and $r$, such that $0 \leq r \leq q-1$, and for any integer $i$
$$
  \begin{aligned}
&\prod_{n \geq 0 }\left(\frac{\left(qn+r+\frac{3}{2}\frac{2r+1}{2q}+2qi+1\right)\left(q(n+1)-r-\frac{3}{2}\frac{2r+1}{2q}+2qi+1\right)}{\left(qn+r+1\right)\left(q(n+1)-r+1\right)}\right)^{\frac{a_n+1}{2}}\\
&\times\prod_{ n \geq 0 }\left(\frac{\left(qn+r+\frac{3}{2}\frac{2r+1}{2q}+2qi-\frac{1}{2}\right)\left(q(n+1)-r-\frac{3}{2}\frac{2r+1}{2q}+2qi-\frac{1}{2}\right)}{\left(qn+r-\frac{1}{2}\right)\left(q(n+1)-r-\frac{1}{2}\right)}\right)^{\frac{1-a_n}{2}}\\
&=2cos(\frac{2r+1}{2q}\pi)
    \end{aligned}
$$

\section{Notation}
Let $(a_n)_{n \in \mathbf{N}}$ be the Thue-Morse sequence satisfying $a_0=1$ and $a_1=-1$ and let $\mathbf{R}^*$ be the free monoid of $\mathbf{R}$ generated by concatenation.

\begin{definition}
Let $\phi: \left\{1,-1\right\} \to \mathbf{R}^*$ be a morphism and $q$ be an integer. $\phi$ is called a $q$-substitution if the length of $\phi(-1)$ and of $\phi(1)$ are the same, and both equal $q$. The Thue-Morse sequence $(a_n)_{n \in \mathbf{N}}$ can be defined as the fixed-point of a $2$-substitution $\phi: \phi(-1)=-1,1$ and $\phi(1)=1,-1$ and with the initial point $a_0=1$.
\end{definition}

\begin{definition}
A $q$-substitution $\phi$ is called alternative if it satisfies that $\phi(1)=l_0,l_1,l_2,...,l_{q-1}$ and $\phi(-1)=-l_0,-l_1,-l_2,...,-l_{q-1}$; and similarly, $\phi$ is called periodic if it satisfies that $\phi(1)=\phi(-1)=l_0,l_1,l_2,...,l_{q-1}$.
\end{definition}

\begin{definition}
Let $\phi$ be a $q$-substitution. Let us define the $\phi$-Thue-Morse sequence to be the image of the Thue-More sequence $(a_n)_{n \in \mathbf{N}}$ under the morphism $\phi$: $\phi((a_n)_{n \in \mathbf{N}})$, we denote this sequence by $(\phi(a)_n)_{n \in \mathbf{N}}$. 
\end{definition}

\begin{example}
Let $q$ be a positive integer, we define $T_q$ a $q$-substitution such that $T_q(1)=\smash{\underbrace{1,1,...,1}_{q\; \text{times}}}$ and $T_q(-1)=\smash{\underbrace{-1,-1,...,-1}_{q \; \text{times}}}$, then the sequence $(T_q(a)_n)_{n \in \mathbf{N}}$ should be in the form:
$$\smash{\underbrace{1,1,...,1}_{q \; \text{times}}}\smash{\underbrace{-1,-1,...,-1}_{q \; \text{times}}}\smash{\underbrace{-1,-1,...,-1}_{q \; \text{times}}}\smash{\underbrace{1,1,...,1}_{q \; \text{times}}}...$$
$$$$
We call this sequence a $q$-stuttered Thue-Morse sequence.
\end{example}

\begin{definition}
Let $\phi$ and $\psi$ be two $q$-substitutions $\left\{1,-1\right\} \to \mathbf{R}^q$ such that $\phi(1)=l_0,l_1,l_2,...,l_{q-1}$; $\phi(-1)=r_0,r_1,r_2,...,r_{q-1}$ and $\psi(1)=l'_0,l'_1,l'_2,...,l'_{q-1}$; $\psi(-1)=r'_0,r'_1,r'_2,...,r'_{q-1}$.Let $\lambda$ be a real number.\\
Let us define $\phi\shuffle\psi$ to be a $2q$-substitution $\left\{1,-1\right\} \to \mathbf{R}^{2q}$ such that $\phi\shuffle\psi(1)=l_0,l'_0,l_1,l'_1,l_2,l'_2...,l_{q-1},l'_{q-1}$and $\phi\shuffle\psi(-1)=r'_0,r_0,r'_1,r_1,r'_2,r_2,...,r'_{q-1},r_{q-1}$;\\
let us define $\phi+\psi$ to be a $q$-substitution $\left\{1,-1\right\} \to \mathbf{R}^{q}$ such that $(\phi+\psi)(1)=l_0+l'_0,l_1+l'_1l_2+l'_2...,l_{q-1}+l'_{q-1}$and $(\phi+\psi)(-1)=r_0+r'_0,r_1+r'_1,r_2+r'_2,...,r_{q-1}+r'_{q-1}$;\\
and let us define $\lambda\phi$ to be a $q$-substitution $\left\{1,-1\right\} \to \mathbf{R}^{q}$ such that $\lambda\phi(1)=\lambda l_0,\lambda l_1, \lambda l_2,...,\lambda l_{q-1}$and $\lambda\phi(-1)=\lambda r_0,\lambda r_1, \lambda r_2,...,\lambda r_{q-1}$;\\
 
\end{definition}

\section{ $\phi$-Thue-Morse sequences and their combined sequences}

Let $\phi$ be an alternative $q$-substitution, $\psi$ be a periodic $q$-substitution such that $\phi(1)=l_0,l_1,l_2,...,l_{q-1}$,$\psi(1)=s_0,s_1,s_2,...,s_{q-1}$ and $s_k+s_{q-1-k}=0$, $l_k=l_{q-1-k}$ for all $k$, $0\leq k \leq q-1$. Let $(T_k(a)_n)_{n \in \mathbf{n}}$ be the $k$-stuttered Thue-Morse sequence. In this section we consider two infinite products:
$$I1=\prod_{n=0}^{\infty}\left(1+\frac{(\phi+T_q)(a)_n}{2n+1}\right),$$
$$I2=\prod_{n=0}^{\infty}\left(1+\frac{(2\psi\shuffle\phi+T_{2q})(a)_n}{2n+1}\right).$$
We will compute the value of $I2/ I1$, if it is well defined. This value will be approached by a sequence of real numbers.

To do so, let us firstly verify that $I1$ and $I2$ are both well defined.

\begin{proposition}
For each alternative $q$-substitution $\phi: \left\{1,-1\right\} \to \mathbf{N}^q$, the following number is well defined:
$$I(\phi)=\prod_{n=0}^{\infty}\left(1+\frac{\phi(a)_n}{2n+1}\right);$$
\end{proposition}

\begin{proof}
Let us consider the following infinite product $\prod_{n=0}^{\infty}\left(1+\frac{a_nl}{kn+b}\right)$. Because of the boundness of the sequence $(a_n)_{n \in \mathbf{N}}$, $1+\frac{a_nl}{kn+b}>0$ for all $n$ large enough, say $n>N_0$. So that
\begin{equation}
  \begin{aligned}
        \log\left(\prod_{n>N_0}^{\infty}\left(1+\frac{a_nl}{kn+b}\right)\right)&=\sum_{n>N_0}^{\infty}\log\left(1+\frac{a_nl}{kn+b}\right)\\
        &=\sum_{n>N_0}^{\infty}\log\left((1+\frac{a_{2n}l}{2kn+b})(1+\frac{a_{2n+1}l}{2kn+b+k})\right)\\
        &=\sum_{n>N_0}^{\infty}\log\left(1+\frac{a_{n}kl-l^2}{(2kn+b)(2kn+b+k)}\right),\\
    \end{aligned}
\end{equation}
and the last sum converges.
To conclude the proof, it is enough to state that 
$$I(\phi)=\prod_{i=0}^{q-1}\prod_{n=0}^{\infty}\left(1+\frac{a_nl_i}{2(qn+i)+1}\right).$$
\end{proof}

\begin{proposition}
Let $\phi$ be an alternative $q$-substitution $\phi: \left\{1,-1\right\} \to \mathbf{N}^q: \phi(1)=l_0,l_1,l_2,...,l_{q-1} \; \phi(-1)=-l_0,-l_1,-l_2,...,-l_{q-1}$ and $\psi$ be a periodic $q$-substitution $\psi: \left\{1,-1\right\} \to \mathbf{N}^q: \psi(1)=\psi(-1)=s_0,s_1,s_2,...,s_{q-1}$, such that $s_k+s_{q-1-k}=0$ for all $k$, $0 \leq k \leq q-1$, then the following number is well defined:
$$I(\psi\shuffle \phi+T_{2q})=\prod_{n=0}^{\infty}\left(1+\frac{(\psi\shuffle \phi+T_{2q})(a)_n}{2n+1}\right);$$
\end{proposition}

\begin{proof}
If $1+\frac{(\psi\shuffle \phi+T_{2q})(a)_n}{2n+1} > 0$ for all integer $n \geq N_0q$, then
\begin{equation}
  \begin{aligned}
&\log\left(\prod_{n=N_0q}^{\infty}\left(1+\frac{(\psi\shuffle \phi+T_{2q})(a)_n}{2n+1}\right)\right)\\
&=\sum_{N_0 \leq j \leq \infty}\sum_{0\leq i \leq q-1}\left(\log(1+\frac{s_i+a_j}{4jq+4i+1}+\log(1+\frac{l_i+a_j}{4jq+4i+3}))\right)\\
&+\sum_{N_0 \leq j \leq \infty}\sum_{0\leq i \leq q-1}\left(\log(1+\frac{-l_i-a_j}{4jq+2q+4i+1})+\log(1+\frac{s_i-a_j}{4jq+2q+4i+3})\right)\\
    \end{aligned}
\end{equation}
Now let us prove that the above infinite sums converge. Firstly, when $j$ is large,
$$\log(1+\frac{l_i+a_j}{4jq+4i+1})+\log(1+\frac{-l_i-a_j}{4jq+2q+4i+1})=O(\frac{l_i+a_j}{4jq+4i+1}-\frac{l_i+a_j}{4jq+2q+4i+1})=O(\frac{1}{n^2}),$$
so that $$\sum_{N_0 \leq j \leq \infty}\sum_{0\leq i \leq q-1}\left(\log(1+\frac{l_i+a_j}{4jq+4i+1})+\log(1+\frac{-l_i-a_j}{4jq+2q+4i+1})\right)$$ converges.\\
Secondly, when $j$ is large,
\begin{equation}
  \begin{aligned}
&\log(1+\frac{s_i+a_j}{4jq+4i+1})+\log(1+\frac{s_{q-1-i}+a_j}{4jq+4(q-1-i)+1})\\
&+\log(1+\frac{s_i-a_j}{4jq+2q+4i+3})+\log(1+\frac{s_{q-1-i}-a_j}{4jq+2q+4(q-1-i)+3})\\
&=O(\frac{s_i+a_j}{4jq+4i+1}+\frac{-s_{i}+a_j}{4jq+4(q-1-i)+1}+\frac{s_i-a_j}{4jq+2q+4i+3}+\frac{-s_{i}-a_j}{4jq+2q+4(q-1-i)+3})\\
&=O(\frac{1}{n^2})
    \end{aligned}
\end{equation}
so that
\begin{equation}
  \begin{aligned}
&2\sum_{N_0 \leq j \leq \infty}\sum_{0\leq i \leq q-1}\left(\log(1+\frac{s_i+a_j}{4jq+4i+1})+\log(1+\frac{s_i-a_j}{4jq+2q+4i+1})\right)\\
&=\sum_{N_0 \leq j \leq \infty}\sum_{0\leq i \leq q-1}\left(\log(1+\frac{s_i+a_j}{4jq+4i+1})+\log(1+\frac{s_{q-1-i}+a_j}{4jq+4(q-1-i)+1}))\right)\\
&+\sum_{N_0 \leq j \leq \infty}\sum_{0\leq i \leq q-1}\left(\log(1+\frac{s_i-a_j}{4jq+2q+4i+3})+\log(1+\frac{s_{q-1-i}-a_j}{4jq+2q+4(q-1-i)+3}))\right)\\
&=\sum_{N_0 \leq j \leq \infty}\sum_{0\leq i \leq q-1}O(\frac{1}{n^2})
    \end{aligned}
\end{equation}
converges;

To conclude, infinite product $I(\psi\shuffle \phi+T_{2q})$ is well defined.
\end{proof}

Now, let us firstly define two sequences of real numbers $(r_n)_{n \in \mathbf{N_+}}$ and $(r'_n)_{n \in \mathbf{N_+}}$.
For each $n \geq 1$:
$$r_n=\prod_{k=0}^{\infty}\prod_{j=0}^{q2^n-1}(1+\frac{(\phi+T_q)(a)_j}{q2^{n+1}k+2j+1});$$
$$r'_n=\prod_{k=0}^{\infty}\prod_{j=0}^{q2^{n+1}-1}(1+\frac{(2\psi\shuffle\phi+T_{2q})(a)_j}{q2^{n+2}k+2j+1}).$$

These sequences are both well defined and we can easily check that $\lim_{n \to \infty} r_n=I1$ and $\lim_{n \to \infty} r'_n=I2$. The first convergence is from the fact that, for each $n$, the $q2^n$ terms with $k=0$ in $r_n$ coincide with the first $q2^n$ terms in the infinite product of $I1$, and $\prod_{k=1}^{\infty}\prod_{j=0}^{q2^n-1}(1+\frac{(\phi(a)+T_q(a))_j}{q2^{n+1}k+2j+1})$ converge to $1$ when $n \to \infty$. And a similar argument works to prove that $\lim_{n \to \infty} r'_n=I2$.

Here we make a first simplification on the expressions of $r_n$ and $r'_n$.

\begin{lemma}
Let $\phi$ be an alternative $q$-substitution, $\psi$ be a periodic $q$-substitution and $(T_k(a)_n)_{n \in \mathbf{n}}$ be the $k$-stuttered Thue-Morse sequence such that $\phi(1)=l_0,l_1,l_2,...,l_{q-1}$, $\phi(-1)=-l_0,-l_1,-l_2,...,-l_{q-1}$ and $\psi(1)=\psi(-1)=s_0,s_1,s_2,...,s_{q-1}$ under the consdition $s_k+s_{q-1-k}=0$, $l_k=l_{q-1-k}$ for all $k$, $0\leq k \leq q-1$, then, with all notation defined as above,

$$r_{2n+1}=\frac{1}{\sqrt{2}}\prod_{j=0}^{q2^{2n}-1}2\sin(\pi (\frac{2j+1}{q2^{2n+2}}+((\phi+T_q)(a)_j)\frac{1}{q2^{2n+2}}));$$
$$r'_{2n+2}=\frac{1}{\sqrt{2}}\prod_{j=0}^{q2^{2n+1}-1}2\sin(\pi (\frac{2j+1}{q2^{2n+3}}+((2\psi\shuffle\phi+T_{2q})(a)_j)\frac{1}{q2^{2n+3}})).$$

\end{lemma}

\begin{proof}
From Lemma 1 in \cite{DILCHER201843}, for each integer $n \geq 2$, taking $f(j)=-\phi(a)_j-T_q(a)_j$, $a=\frac{1}{q2^{n+1}}$ and $z_j=\frac{2j+1}{q2^{n+1}}$, we have $\sum_{j=0}^{q2^n-1}f(j)=0$, so that

\begin{equation}
r_n=\prod_{j=0}^{q2^n-1}\frac{\Gamma(\frac{2j+1}{q2^{n+1}})}{\Gamma(\frac{2j+1}{q2^{n+1}}+((\phi+T_q)(a)_j)\frac{1}{q2^{n+1}})}.
\end{equation}
Lemma 2 in \cite{DILCHER201843} gives us 
\begin{equation}
\prod_{j=0}^{q2^n-1}\Gamma(\frac{2j+1}{q2^{n+1}})=\frac{\prod_{j=0}^{2^{n+1}-1}\Gamma(\frac{j}{q2^{n+1}})}{\prod_{j=0}^{2^n-1}\Gamma(\frac{j}{q2^{n}})}=\frac{\frac{(2\pi)^{(q2^{n+1}-1)/2}}{\sqrt{q2^{n+1}}}}{\frac{(2\pi)^{(q2^{n}-1)/2}}{\sqrt{q2^{n}}}}=\frac{1}{\sqrt{2}}(2\pi)^{q2^{n-1}}.
\end{equation}
From the facts that $\forall 0 \leq j\leq q2^{2n+1}-1, \; \phi(a)_j=-\phi(a)_{q2^{2n+1}-1-j}$, $T_q(a)_j=-T_q(a)_{q2^{2n+1}-1-j}$ and $\Gamma(z)\Gamma(1-z)=\frac{\pi}{\sin(\pi z)}$, we have
\begin{equation}
\prod_{j=0}^{q2^{2n+1}-1}\Gamma(\frac{2j+1}{q2^{2n+2}}+((\phi+T_q)(a)_j)\frac{1}{q2^{2n+2}})=\prod_{j=0}^{q2^{2n}-1}\frac{\pi}{\sin(\pi (\frac{2j+1}{q2^{2n+2}}+((\phi+T_q)(a)_j)\frac{1}{q2^{2n+2}}))}.
\end{equation}

As a result, \begin{equation}
r_{2n+1}=\frac{1}{\sqrt{2}}\prod_{j=0}^{q2^{2n}-1}2\sin(\pi (\frac{2j+1}{q2^{2n+2}}+((\phi+T_q)(a)_j)\frac{1}{q2^{2n+2}})).
\end{equation}

Similarly, from Lemma 1 in \cite{DILCHER201843}, for each integer $n \geq 2$, taking $f(j)=-2\psi\shuffle\phi(a)_j-T_{2q}(a)_j$, $a=\frac{1}{q2^{n+2}}$ and $z_j=\frac{2j+1}{q2^{n+2}}$, we have $\sum_{j=0}^{q2^{n+2}-1}f(j)=0,$ so that

\begin{equation}
r'_n=\prod_{j=0}^{q2^{n+1}-1}\frac{\Gamma(\frac{2j+1}{q2^{n+2}})}{\Gamma(\frac{2j+1}{q2^{n+2}}+((2\psi\shuffle\phi+T_{2q})(a)_j)\frac{1}{q2^{n+2}})}.
\end{equation}
By the same argument as above,
\begin{equation}
\prod_{j=0}^{q2^{n+1}-1}\Gamma(\frac{2j+1}{q2^{n+2}})=\frac{1}{\sqrt{2}}(2\pi)^{q2^{n}}.
\end{equation}
From the fact that $\forall 0 \leq j\leq q2^{2n+2}-1, \; \psi\shuffle\phi(a)_j=-\psi\shuffle\phi(a)_{q2^{2n+2}-1-j}$, $T_{2q}(a)_j=-T_{2q}(a)_{q2^{2n+2}-1-j}$, we have
\begin{equation}
\prod_{j=0}^{q2^{2n+2}-1}\Gamma(\frac{2j+1}{q2^{2n+3}}+((2\psi\shuffle\phi+T_{2q})(a)_j)\frac{1}{q2^{2n+3}})=\prod_{j=0}^{q2^{2n+1}-1}\frac{\pi}{\sin(\pi (\frac{2j+1}{q2^{2n+3}}+((2\psi\shuffle\phi+T_{2q})(a)_j)\frac{1}{q2^{2n+3}}))}.
\end{equation}
As a result,
\begin{equation}
r'_{2n+2}=\frac{1}{\sqrt{2}}\prod_{j=0}^{q2^{2n+1}-1}2\sin(\pi (\frac{2j+1}{q2^{2n+3}}+((2\psi\shuffle\phi+T_{2q})(a)_j)\frac{1}{q2^{2n+3}})).
\end{equation}

\end{proof}

\begin{lemma}
Let $(\phi(a)_n)_{n \in \mathbf{N}}$ be a sequence defined by an alternative morphism as above, such that $\phi(1)=l_0,l_1,l_2,...,l_{q-1}$ and $l_k=l_{q-1-k}$ for all $k$, $0\leq k \leq k-1$. Then 
\begin{equation}
r_{2n+1}=\frac{1}{\sqrt{2}}\prod_{0 \leq j \leq 2^{2n}-1 \atop a_{j}=1}\prod_{0 \leq k \leq q}2\cos (\frac{(qj+k-l_{q-1-k}/2)\pi}{q2^{2n+1}})\prod_{0 \leq j \leq 2^{2n}-1 \atop a_{j}=-1}\prod_{0 \leq k \leq q}2\sin (\frac{(qj+k-l_k/2)\pi}{q2^{2n+1}})
\end{equation}
.
\end{lemma}

\begin{proof}
From the fact that , $a_j=a_{2^{2n}-j-1}$, we can prove that for all $k$, $0 \leq k \leq q-1$,  
$$\phi(a)_{qj+k}=\phi(a)_{q(2^{2n}-j-1)+k}=a_jl_k;$$
$$T_q(a)_{qj+k}=T_q(a)_{q(2^{2n}-j-1)+k}=a_j.$$
As a result, if $a_j=1$ then  for all $k$ such that $0 \leq k \leq q-1$,
\begin{equation}
  \begin{aligned}
\sin(\pi (\frac{2(qj+k)+1}{q2^{2n+2}}+\frac{(\phi+T_q)(a)_{qj+k}}{q2^{2n+2}})&= \sin(\frac{(2(qj+k)+1+l_k+1)\pi}{q2^{2n+2}})\\
&=\sin(\frac{(qj+k+1+l_k/2)\pi}{q2^{2n+1}})\\
&=\cos(\frac{(q2^{2n}-qj-k-1-l_{k}/2)\pi}{q2^{2n+1}})\\
&=\cos(\frac{(q(2^{2n}-j-1)+q-1-k-l_{k}/2)\pi}{q2^{2n+1}})
    \end{aligned}
\end{equation}
 and
\begin{equation}
  \begin{aligned}
\sin(\pi (\frac{2(q(2^{2n}-j-1)+k)+1}{q2^{2n+2}}+\frac{(\phi(a)+T_q)(a)_{q(2^{2n}-j-1)+k}}{q2^{2n+2}})&=\sin(\frac{(2(q(2^{2n}-j-1)+k)+1+l_k+1)\pi}{q2^{2n+2}})\\
&=\sin(\frac{(q(2^{2n}-j-1)+k+1+l_k/2)\pi}{q2^{2n+1}})\\
&=\cos(\frac{(qj+q-k-1-l_{k}/2)\pi}{q2^{2n+1}})
    \end{aligned}
\end{equation}
And if $a_j=-1$ then  for all $k$ such that $0 \leq k \leq q-1$,
$$\phi(a)_{qj+k}=-l_k;$$
$$T_q(a)_{qj+k}=-1.$$
\begin{equation}
  \begin{aligned}
\sin(\pi (\frac{2(qj+k)+1}{q2^{2n+2}}+\frac{(\phi+T_q)(a)_{qj+k}}{q2^{2n+2}})&= \sin(\frac{(2(qj+k)+1-l_k-1)\pi}{q2^{2n+2}})\\
&=\sin(\frac{(qj+k-l_k/2)\pi}{q2^{2n+1}})
    \end{aligned}
\end{equation}
So $$r_{2n+1}=\frac{1}{\sqrt{2}}\prod_{0 \leq j \leq 2^{2n}-1 \atop a_{j}=1}\prod_{0 \leq k \leq q}2\cos (\frac{(qj+k-l_{q-1-k}/2)\pi}{q2^{2n+1}})\prod_{0 \leq j \leq 2^{2n}-1 \atop a_{j}=-1}\prod_{0 \leq k \leq q}2\sin (\frac{(qj+k-l_k/2)\pi}{q2^{2n+1}}).$$
\end{proof}

In an analogous way, we reformulate the expression of $r'_{2n+2}$ in the same form.

\begin{lemma}
Let $((2\psi\shuffle\phi+T_{2q})(a)_n)_{n \in \mathbf{N}}$ be a sequence defined as above, such that $\phi(1)=l_0,l_1,l_2,...,l_{q-1}$,$\psi(1)=s_0,s_1,s_2,...,s_{q-1}$ and $s_k+s_{q-1-k}=0$, $l_k=l_{q-1-k}$ for all $k$, $0\leq k \leq q-1$.Then 
\begin{equation}
  \begin{aligned}
r'_{2n+2}&=\frac{1}{\sqrt{2}}\prod_{0 \leq j \leq 2^{2n}-1 \atop a_{j}=1}\prod_{0 \leq k \leq q}2\cos (\frac{(qj+k-l_{q-1-k}/2)\pi}{q2^{2n+1}})\prod_{0 \leq j \leq 2^{2n}-1 \atop a_{j}=-1}\prod_{0 \leq k \leq q}2\sin (\frac{(qj+k-l_k/2)\pi}{q2^{2n+1}})\\
&\times \prod_{0 \leq j \leq 2^{2n}-1 \atop a_{j}=1}\prod_{0 \leq k \leq q}2\cos (\frac{(2(qj+k)+1-s_{k})\pi}{q2^{2n+2}})\prod_{0 \leq j \leq 2^{2n}-1 \atop a_{j}=-1}\prod_{0 \leq k \leq q}2\sin (\frac{(2(qj+k)+1-s_k)\pi}{q2^{2n+2}})
    \end{aligned}
\end{equation}
\end{lemma}

\begin{proof}

If $a_j=1$, then $a_{2^{2n}-j-1}=1$, so that for all $k$, $0 \leq k \leq q-1$,  
$$(2\psi\shuffle\phi+T_{2q})(a)_{2(qj+k)}=(2\psi\shuffle\phi+T_{2q})(a)_{2(q2^{2n+1}-qj)+2k}=2s_k+1;$$
$$(2\psi\shuffle\phi+T_{2q})(a)_{2(qj+k)+1}=(2\psi\shuffle\phi+T_{2q})(a)_{2(q2^{2n+1}-2qj)+2k+1}=2l_{k}+1.$$
By similar calculations as in (14) and (15), we have
$$\sin(\pi (\frac{2(2qj+2k)+1}{q2^{2n+3}}+\frac{(2\psi\shuffle\phi+T_{2q})(a)_{qj+k}}{q2^{2n+3}})=\cos(\frac{(2q(2^{2n-1}-1-j)+(2q-1-2k)-s_{k})\pi}{q2^{2n+2}})$$
$$\sin(\pi (\frac{2(2qj+2k+1)+1}{q2^{2n+3}}+\frac{(2\psi\shuffle\phi+T_{2q})(a)_{qj+k}}{q2^{2n+3}})=\cos(\frac{(q(2^{2n-1}-1-j)+q-1-k-l_{k}/2)\pi}{q2^{2n+1}})$$
and 
$$\sin(\pi (\frac{2(2q(2^{2n}-j-1)+2k)+1}{q2^{2n+3}}+\frac{(2\psi\shuffle\phi+T_{2q})(a)_{2q(2^{2n}-j-1)+2k}}{q2^{2n+3}})=\cos(\frac{(2qj+2q-2k-1-s_{k})\pi}{q2^{2n+2}})$$
$$\sin(\pi (\frac{2(2q(2^{2n}-j-1)+2k+1)+1}{q2^{2n+3}}+\frac{(2\psi\shuffle\phi+T_{2q})(a)_{2q(2^{2n}-j-1)+2k}}{q2^{2n+3}})=\cos(\frac{(qj+q-k-1+l_{q-1-k}/2)\pi}{q2^{2n+1}})$$
And if $a_j=-1$, then $a_{2^{2n}-j-1}=-1$,so that for all $k$, $0 \leq k \leq q-1$,  
$$(2\psi\shuffle\phi+T_{2q})(a)_{2(qj+k)}=(2\psi\shuffle\phi+T_{2q})(a)_{2(q2^{2n+1}-qj)+2k}=2l_k-1;$$
$$(2\psi\shuffle\phi+T_{2q})(a)_{2(qj+k)+1}=(2\psi\shuffle\phi+T_{2q})(a)_{2(q2^{2n+1}-2qj)+2k+1}=2s_{k}-1.$$
By a similar calculation as in (16), we have
$$\sin(\pi (\frac{2(2qj+2k)+1}{q2^{2n+3}}+\frac{(2\psi\shuffle\phi+T_{2q})(a)_{2qj+2k}}{q2^{2n+3}})=\sin(\frac{(qj+k-l_k/2)\pi}{q2^{2n+1}})$$
$$\sin(\pi (\frac{2(2qj+2k+1)+1}{q2^{2n+3}}+\frac{(2\psi\shuffle\phi+T_{2q})(a)_{2qj+2k+1}}{q2^{2n+3}})=\sin(\frac{(2qj+2k+1-s_k)\pi}{q2^{2n+2}})$$
Combining all above equalities, we prove the lemma. 
\end{proof}

\begin{theorem}
Let $I1$ and $I2$ be the two infinite products defined as above, with an alternative $q$-substitution $\phi$ and a periodic $q$-substitution $\psi$ such that $\phi(1)=l_0,l_1,...,l_{q-1}$ and $\psi(1)=s_0,s_1,...,s_{q-1}$. If additionally $s_k=-s_{q-1-k}$ and $l_k=l_{q-1-k}$ for all $k$, $0 \leq k \leq q-1$, and $I1\neq 0$ then  
$$I_2/I_1=(-1)^{\sum_{0\leq k \leq q-1}t_k}2^{q/2}\left(\prod_{0 \leq k \leq q-1}\sin(\frac{2k+1-2s_k}{2q}\pi)\right)^{\frac{1}{2}},$$ 
where $t_k$ is the number of positive intergers smaller than $\frac{2s_k-2k-1}{2q}$. Furthermore, the above equality does not depend on the choice of the $q$-substitution $\phi$.
\end{theorem}

\begin{proof}
As we know that $r_n \to I1$ and $r'_n \to I2$ when $n \to \infty$. So $I2/I1=\lim_{n \to \infty}r'_{2n+2}/r_{2n+1}$. From Lemma 1 and Lemma 2, for $n \geq 1$,
\begin{equation}
  \begin{aligned}
r'_{2n+2}/r_{2n+1}&=  \prod_{0 \leq j \leq 2^{2n}-1 \atop a_{j}=1}\prod_{0 \leq k \leq q-1}2\cos (\frac{(2(qj+k)+1-s_{k})\pi}{q2^{2n+2}})\prod_{0 \leq j \leq 2^{2n}-1 \atop a_{j}=-1}\prod_{0 \leq k \leq q-1}2\sin (\frac{(2(qj+k)+1-s_k)\pi}{q2^{2n+2}})\\
&= \prod_{0 \leq j \leq 2^{2n-1}-1 \atop a_{j}=1}\prod_{0 \leq k \leq q-1}2\cos (\frac{(2(qj+k)+1-s_{k})\pi}{q2^{2n+2}})2\cos (\frac{(2(q(2^{2n-1}-1-j)+2(q-1-k)+1+s_{k})\pi}{q2^{2n+2}})\\
&\times\prod_{0 \leq j \leq 2^{2n-1}-1 \atop a_{j}=1}\prod_{0 \leq k \leq q-1}2\cos (\frac{(2(qj+k)+1-s_{k})\pi}{q2^{2n+2}})2\sin (\frac{(2(q(2^{2n-1}-1-j)+2(q-1-k)+1+s_{k})\pi}{q2^{2n+2}})\\ 
&= \prod_{0 \leq j \leq 2^{2n-1}-1 \atop a_{j}=1}\prod_{0 \leq k \leq q-1}2(\cos (\frac{(2(qj+k)+1-2s_{k})\pi}{q2^{2n+1}}-\frac{\pi}{2})+\cos(\frac{\pi}{2}))\\
&\times\prod_{0 \leq j \leq 2^{2n-1}-1 \atop a_{j}=-1}\prod_{0 \leq k \leq q-1}2(\cos (\frac{(2(qj+k)+1-2s_k)\pi}{q2^{2n+1}}-\frac{\pi}{2})-\cos(\frac{\pi}{2}))\\
&= \prod_{0 \leq j \leq 2^{2n-1}-1}\prod_{0 \leq k \leq q-1}2\sin (\frac{(2(qj+k)+1-2s_{k})\pi}{q2^{2n+1}})
    \end{aligned}
\end{equation}
The second equality is from the hypothesis that $s_k=s_{q-1-k}$ and $a_j=1$ if and only if $a_{2^n-1-j}=1$.
The sign of the number $r'_{2n+2}/r_{2n+1}$ is $(-1)^{\sum_{0\leq k\leq q-1}t^{(n)}_k}$, where $t^{(n)}_k$ is number of integers smaller than $\min\left\{2^{2n-1}-1, \frac{2s_k-2k-1}{2q}\right\}$.
Now using once more the hypothesis that $s_k=s_{q-1-k}$,
\begin{equation}
  \begin{aligned}
(r'_{2n+2}/r_{2n+1})^2&= \prod_{0 \leq j \leq 2^{2n-1}-1}\prod_{0 \leq k \leq q-1}4\sin^2 (\frac{(2(qj+k)+1-2s_{k})\pi}{q2^{2n+1}})\\
&= \prod_{0 \leq j \leq 2^{2n-1}-1}\prod_{0 \leq k \leq q-1}2\sin (\frac{(2(qj+k)+1-2s_{k})\pi}{q2^{2n+1}})2\sin (\frac{(2(qj+(q-1-k))+1+2s_{k})\pi}{q2^{2n+1}})\\
&=\prod_{0 \leq j \leq 2^{2n-1}-1}\prod_{0 \leq k \leq q-1}2\sin (\frac{(2(qj+k)+1-2s_{k})\pi}{q2^{2n+1}})\\
&\times\prod_{1 \leq j \leq 2^{2n-1}}\prod_{0 \leq k \leq q-1}2\sin (\frac{(2(qj-k)-1+2s_{k})\pi}{q2^{2n+1}})\\
&=\prod_{1 \leq j \leq 2^{2n-1}-1}\prod_{0 \leq k \leq q-1}2\sin (\frac{(2qj+2k+1-2s_k)\pi}{q2^{2n+1}})2\sin (\frac{(2q(2^{2n-1}-j)-2k-1+2s_k)\pi}{q2^{2n+1}})\\
&\times\prod_{0 \leq k \leq q-1}2\sin (\frac{(2k+1-2s_{k})\pi}{q2^{2n+1}})2\sin (\frac{(-2k-1+2s_{k})\pi}{q2^{2n+1}}+\frac{\pi}{2})\\
&=\prod_{1 \leq j \leq 2^{2n-1}-1}\prod_{0 \leq k \leq q-1}2(\sin (\frac{(2qj+2k+1-2s_k)\pi}{q2^{2n}}))\times\prod_{0 \leq k \leq q-1}2\sin (\frac{(2k+1-2s_{k})\pi}{q2^{2n}})\\
&=\prod_{1 \leq j \leq 2^{2n-1}}\prod_{0 \leq k \leq q-1}2(\sin (\frac{(2qj+2k+1-2s_k)\pi}{q2^{2n}}))\\
    \end{aligned}
\end{equation}
Now use the equality that for all real number $x$ and integer $n$, 
\begin{equation}\sin(nx)=\frac{1}{2}\prod_{j=0}^{n-1}2\sin(x+\frac{j\pi}{n})\end{equation}
( for the proof, see for example \cite{22350}), we have 
\begin{equation}
  \begin{aligned}
(r'_{2n+2}/r_{2n+1})^2&=\prod_{1 \leq j \leq 2^{2n-1}}\prod_{0 \leq k \leq q-1}2(\sin (\frac{(2qj+2k+1-2s_k)\pi}{q2^{2n}}))\\
&=2^q\prod_{0 \leq k \leq q-1}\sin(\frac{2k+1-2s_k}{2q}\pi)
    \end{aligned}
\end{equation}
The last product in the above equality is a positive constant because of the fact that
\begin{equation}
  \begin{aligned}
\sin(\frac{2k+1-2s_k}{2q}\pi)\sin(\frac{2(2q-1-k)+1-2s_{q-1-k}}{2q}\pi)&=\sin(\frac{2k+1-2s_k}{2q}\pi)\sin(\pi-\frac{2k-1-2s_{k}}{2q}\pi)\\
&=\sin^2(\frac{2k+1-2s_k}{2q}\pi)
    \end{aligned}
\end{equation}
To conclude, $$I_2/I_1=\lim_{n\to \infty}r'_{2n+2}/r_{2n+1}=(-1)^{\sum_{0\leq k \leq q-1}t_k}2^{q/2}\left(\prod_{0 \leq k \leq q-1}\sin(\frac{2k+1-2s_k}{2q}\pi)\right)^{\frac{1}{2}}.$$
\end{proof}

\begin{remark}
In the proof of Theorem 1, we do not use the hypothesis that $I1 \neq 0$. In fact, from Proposition 1, $I1=0$ if and only if there exists  a $k$ such that $1+\frac{(\phi+T_q)(a)_k}{2k+1}=0$. However, in this case, we can check easily that $1+\frac{(2\psi\shuffle\phi+T_{2q})(a)_{2k}}{4k+1}=0$. So when calculating $r'_{2n+2}/r_{2n+1}$, the zeros in $r_{2n+1}$ ``cancel" with some zeros in $r'_{2n+2}$, which makes the calculation still ``work" in the case that $I1=0$.
\end{remark}

\begin{corollary}
For any positive integer $q$,
$$\frac{\prod_{n=0}^{\infty}\left(1+\frac{T_{2q}(a)_n}{2n+1}\right)}{\prod_{n=0}^{\infty}\left(1+\frac{T_q(a)_n}{2n+1}\right)}=\sqrt{2}$$
\end{corollary}

\begin{proof}
Let $\phi$ and $\psi$ be 2 $q$-substitution such that $\phi(1)=\phi(-1)=\psi(1)=\psi(-1)=\smash{\underbrace{0,0,...,0}_{q \; \text{times}}}$, then they satisfy the condition in Theorem 1, and  $$I1=\prod_{n=0}^{\infty}\left(1+\frac{(\phi+T_q)(a)_n}{2n+1}\right)=\prod_{n=0}^{\infty}\left(1+\frac{T_q(a)_n}{2n+1}\right),$$
$$I2=\prod_{n=0}^{\infty}\left(1+\frac{(2\psi\shuffle\phi+T_{2q})(a)_n}{2n+1}\right)=\prod_{n=0}^{\infty}\left(1+\frac{T_{2q}(a)_n}{2n+1}\right).$$
So that $\frac{\prod_{n=0}^{\infty}\left(1+\frac{T_{2q}(a)_n}{2n+1}\right)}{\prod_{n=0}^{\infty}\left(1+\frac{T_q(a)_n}{2n+1}\right)}=2^{q/2}\prod_{0\leq k\leq q-1}\sin(\frac{2k+1}{2q}\pi)^{\frac{1}{2}}=(2\sin(\frac{\pi}{2}))^{1/2}$. The last equality is from the equality (20).

\end{proof}

\begin{corollary}
For given integers $q$ and $r$, such that $0 \leq r \leq q-1$, and for any real number $s$,
\begin{equation}
  \begin{aligned}
\prod_{n \geq 0 }\left(\frac{\left(qn+r+s+1\right)\left(q(n+1)-r-s+1\right)}{\left(qn+r+1\right)\left(q(n+1)-r+1\right)}\right)^{\frac{a_n+1}{2}}&\prod_{ n \geq 0 }\left(\frac{\left(qn+r+s-\frac{1}{2}\right)\left(q(n+1)-r-s-\frac{1}{2}\right)}{\left(qn+r-\frac{1}{2}\right)\left(q(n+1)-r-\frac{1}{2}\right)}\right)^{\frac{1-a_n}{2}}\\
&=(-1)^{t(s)}\frac{|\sin(\frac{2r+1-2s}{2q}\pi)|}{\sin(\frac{2r+1}{2q}\pi)}
    \end{aligned}
\end{equation}
where $t(s)$ is the number of positive intergers smaller than $\frac{2s-2r-1}{2q}$ if $s$ is positive. Otherwise,  $t(s)$ is the number of positive intergers smaller than $\frac{-2s-2(q-r)-1}{2q}$. Consequencely, the product of infinite products as above is $2q$-periodic on function of $s$.
\end{corollary}

\begin{proof}
Firstly let $\phi$ and $\psi$ be 2 $q$-substitution such that $\phi(1)=\phi(-1)=0,0,...,0$, and $\psi(1)=\psi(-1)=0,0,...,0,s,0,...,0,-s,0,...,0$ such that the $r$th element in the string $\psi(1)$ is $s$ and the $(q+1-r)$th element is $-s$, then they satisfy the condition in Theorem 1, and  $$I1=\prod_{n=0}^{\infty}\left(1+\frac{(\phi+T_q)(a)_n}{2n+1}\right)=\prod_{n=0}^{\infty}\left(1+\frac{T_q(a)_n}{2n+1}\right),$$
$$I2=\prod_{n=0}^{\infty}\left(1+\frac{(2\psi\shuffle\phi+T_{2q})(a)_n}{2n+1}\right)=\prod_{n=0}^{\infty}\left(1+\frac{T_{2q}(a)_n}{2n+1}\right)\prod_{\substack{n\geq 0 \\ \psi\shuffle\phi(a)_n \neq 0} }\frac{\left(1+\frac{2\psi\shuffle\phi(a)_n+T_{2q}(a)_n}{2n+1}\right)}{\left(1+\frac{T_{2q}(a)_n}{2n+1}\right)}.$$
So that $$I2/I1=(-1)^{t(s)}2^{q/2}\prod_{0\leq k\leq q-1}\sin(\frac{2k+1}{2q}\pi)^{\frac{1}{2}}(\frac{\sin(\frac{2r+1-2s}{2q}\pi)\sin(\pi-\frac{2r+1-2s}{2q}\pi)}{\sin(\frac{2r+1}{2q}\pi)\sin(\pi-\frac{2r+1}{2q}\pi)})^{\frac{1}{2}}=(-1)^{t(s)}2^{1/2}(\frac{|\sin(\frac{2r+1-2s}{2q}\pi)|}{\sin(\frac{2r+1}{2q}\pi)}).$$ 
Furthermore, $$\prod_{\substack{n\geq 0 \\ \psi\shuffle\phi(a)_n \neq 0} }\frac{\left(1+\frac{2\psi\shuffle\phi(a)_n+T_{2q}(a)_n}{2n+1}\right)}{\left(1+\frac{T_{2q}(a)_n}{2n+1}\right)}=(-1)^{t(s)}\frac{|\sin(\frac{2r+1-2s}{2q}\pi)|}{\sin(\frac{2r+1}{2q}\pi)}.$$
Secondly, for given positive integer $N$,
\begin{equation}
  \begin{aligned}
\prod_{\substack{0 \leq n \leq 4qN \\ \psi\shuffle\phi(a)_n \neq 0}}\frac{\left(1+\frac{2\psi\shuffle\phi(a)_n+T_{2q}(a)_n}{2n+1}\right)}{\left(1+\frac{T_{2q}(a)_n}{2n+1}\right)}&=\prod_{\substack{0 \leq n \leq 2N \\ a_n=1} }\frac{\left(1+\frac{2s+1}{2qn+2r+1}\right)\left(1+\frac{-2s+1}{2qn+2(q-r)+1}\right)}{\left(1+\frac{1}{2qn+2r+1}\right)\left(1+\frac{1}{2qn+2(q-r)+1}\right)}\\
&\times\prod_{\substack{0 \leq n \leq 2N \\ a_n=-1} }\frac{\left(1+\frac{2s-1}{2qn+(2r-1)+1}\right)\left(1+\frac{-2s-1}{2qn+(2(q-r)-1)+1}\right)}{\left(1-\frac{1}{2qn+(2r-1)+1}\right)\left(1-\frac{1}{2qn+(2(q-r)-1)+1}\right)}\\
&=\prod_{\substack{0 \leq n \leq 2N \\ a_n=1} }\frac{\left(qn+r+s+1\right)\left(q(n+1)-r-s+1\right)}{\left(qn+r+1\right)\left(q(n+1)-r+1\right)}\\
&\times\prod_{\substack{0 \leq n \leq 2N \\ a_n=-1} }\frac{\left(qn+r+s-\frac{1}{2}\right)\left(q(n+1)-r-s-\frac{1}{2}\right)}{\left(qn+r-\frac{1}{2}\right)\left(q(n+1)-r-\frac{1}{2}\right)}\\
&=\prod_{0 \leq n \leq 2N }\left(\frac{\left(qn+r+s+1\right)\left(q(n+1)-r-s+1\right)}{\left(qn+r+1\right)\left(q(n+1)-r+1\right)}\right)^{\frac{a_n+1}{2}}\\
&\times\prod_{0 \leq n \leq 2N }\left(\frac{\left(qn+r+s-\frac{1}{2}\right)\left(q(n+1)-r-s-\frac{1}{2}\right)}{\left(qn+r-\frac{1}{2}\right)\left(q(n+1)-r-\frac{1}{2}\right)}\right)^{\frac{1-a_n}{2}}\\
    \end{aligned}
\end{equation}
We can check easily that the last two finite products both converge when $N$ tends to infinite, so we conclude the proof.
\end{proof}

\begin{corollary}
For given integers $q$ and $r$, such that $0 \leq r \leq q-1$, for any integer $i$
\begin{equation}
  \begin{aligned}
&\prod_{n \geq 0 }\left(\frac{\left(qn+r+\frac{3}{2}\frac{2r+1}{2q}+2qi+1\right)\left(q(n+1)-r-\frac{3}{2}\frac{2r+1}{2q}+2qi+1\right)}{\left(qn+r+1\right)\left(q(n+1)-r+1\right)}\right)^{\frac{a_n+1}{2}}\\
&\times\prod_{ n \geq 0 }\left(\frac{\left(qn+r+\frac{3}{2}\frac{2r+1}{2q}+2qi-\frac{1}{2}\right)\left(q(n+1)-r-\frac{3}{2}\frac{2r+1}{2q}+2qi-\frac{1}{2}\right)}{\left(qn+r-\frac{1}{2}\right)\left(q(n+1)-r-\frac{1}{2}\right)}\right)^{\frac{1-a_n}{2}}\\
&=2cos(\frac{2r+1}{2q}\pi)
    \end{aligned}
\end{equation}
\end{corollary}

\begin{proof}
By taking $s=\frac{3}{2}\frac{2r+1}{2q}+2qi$ and considering a particular case that $i=0$ in Corollary 2, we check it easily that the two infinite products in (25) are positive. Furthermore, we have
$$\frac{|sin(\frac{2r+1-2s}{2q}\pi)|}{sin(\frac{2r+1}{2q}\pi)}=\frac{|sin(\frac{-2(2r+1)}{2q}\pi)|}{sin(\frac{2r+1}{2q}\pi)}=\frac{2sin(\frac{2r+1}{2q}\pi)cos(\frac{2r+1}{2q}\pi)}{sin(\frac{2r+1}{2q}\pi)}=2cos(\frac{2r+1}{2q}\pi).$$
We prove the corollary for $i=0$. Using the periodicity on $s$, we conclude the proof in the general case.
\end{proof}

\bibliographystyle{alpha}
\bibliography{citations_V4}

\end{document}